\documentclass[11pt,a4paper]{amsart}

\usepackage{amsmath,amssymb,enumerate,amsthm}

\newcommand{\inR}[1]{\,#1\,}
\newcommand{\UP}{\blacktriangle}
\newcommand{\DOWN}{\blacktriangledown}
\newcommand{\Up}{\vartriangle}
\newcommand{\Down}{\triangledown}

\theoremstyle{plain}

\newtheorem{theorem}{Theorem}[section]
\newtheorem{proposition}[theorem]{Proposition}
\newtheorem{lemma}[theorem]{Lemma}
\newtheorem{corollary}[theorem]{Corollary}

\theoremstyle{remark}
\newtheorem{remark}[theorem]{Remark}

\begin{document}
\title{Rough Sets Determined by Quasiorders}
\author{Jouni J{\"a}rvinen}
\address{Jouni J{\"a}rvinen: Department of Information Technology \\
FI-20014 University of Turku, Finland}
\email{Jouni.Jarvinen@utu.fi}
\author{S{\'a}ndor Radeleczki}
\address{S{\'a}ndor Radeleczki: Institute of Mathematics \\ 
University of Miskolc \\ 3515~Miskolc-Egyetemv{\'a}ros \\ Hungary}
\email{matradi@uni-miskolc.hu}
\thanks{The partial support by Hungarian National Research Found 
(Grant No.~T049433/05 and T046913/04) is acknowledged by the second author}
\author{Laura Veres}
\address{Laura Veres: Institute of Mathematics \\ 
University of Miskolc \\ 3515~Mis\-kolc-Egyetemv{\'a}ros \\ Hungary}
\email{lauracicu@yahoo.com}

\subjclass[2000]{Primary 06A06; Secondary  06D10, 06D15, 68T37}

\keywords{Rough set, rough approximations, quasiorder, Alexandrov topology, 
de~Morgan operation, pseudocomplement, completely distributive lattice, Stone lattice,
completely join-irreducible element}

\begin{abstract}
In this paper, the ordered set of rough sets determined by a quasiorder relation 
$R$ is investigated. We prove that this ordered set is a complete, completely
distributive lattice. We show that on this lattice can be defined three different 
kinds of complementation operations,  and we describe its completely join-irreducible elements. 
We also characterize the case in which this lattice is a Stone lattice. 
Our results generalize some results of J.~Pomyka{\l}a and J.~A.~Pomyka{\l}a (1988)
and M.~Gehrke and E.~Walker (1992) in case  $R$ is an equivalence.
\end{abstract}

\maketitle

\section{Introduction} \label{Sec:ONE}

Rough set theory was introduced by Z.~Pawlak in \cite{Pawl82}. His idea was to
develop a formalism for dealing with vague concepts and sets. In rough set
theory it is assumed that our knowledge is restricted by an indiscernibility
relation. An \emph{indiscernibility relation} is an equivalence $E$ such
that two elements of a universe of discourse $U$ are $E$-equivalent if we
cannot distinguish these two elements by their properties known by us. By
the means of an indiscernibility relation $E$, we can partition the elements
of $U$ into three disjoint classes with respect to any set $X \subseteq U$:

\begin{enumerate}[\rm (1)]

\item The elements which are certainly in $X$. These are elements $x \in U$
whose $E$-class $x/E$ is included in $X$.

\item The elements which certainly are not in $X$. These are elements $x \in U
$ such that their $E$-class $x/E$ is included in $X$'s complement $X^c$.

\item The elements which are possibly in $X$. These are elements whose 
$E$-class intersects with both $X$ and $X^c$. In other words, $x/E$ is included
neither in $X$ nor in $X^c$.
\end{enumerate}

Based on this observation, Z.~Pawlak defined the \emph{lower approximation} 
$X^\DOWN$ of $X$ to be the set of those elements $x \in U$ whose $E$-class is
included in $X$. The \emph{upper approximation} $X^\UP$ of $X$ consists of
elements $x \in X$ whose $E$-class intersects with $X$. Then, the sets $X^\DOWN$
and $X^\UP$ can be viewed as sets of elements that belong certainly and
possibly to $X$, respectively. The difference $X^\UP \setminus X^\DOWN$ can be
viewed as the actual area of uncertainty.

Interestingly, we may define another indiscernibility relation, but now
between subsets of $U$. The relation $\equiv$ is based on approximations and
it is called \emph{rough equality}. The sets $X$ and $Y$ are $\equiv$-related 
if both of their approximations are the same, that is, $X^\DOWN =
Y^\DOWN$ and $X^\UP = Y^\UP$. The equivalence classes of $\equiv$ are called 
\emph{rough sets}. Each element in the same rough set looks the same, when
observed through the knowledge given by the indiscernibility relation $E$.
Namely, if $X \equiv Y$, then exactly the same elements belong certainly and
possibly to $X$ and $Y$.

Lattice-theoretical study of rough sets was initiated by T.~B.~Iwi{\'n}ski in 
\cite{Iwin87}. He noticed that rough sets can be represented simply by their
approximations. Hence, the set of rough sets can be defined as 
\begin{equation*}
\textit{RS\/} = \{ (X^\DOWN,X^\UP) \mid X \subseteq U \}. 
\end{equation*}
T.~B.~Iwi{\'n}ski also noted that $\textit{RS\/}$ may be canonically ordered by the
coordinatewise order: 
\begin{equation*}
(X^\DOWN,X^\UP) \leq (Y^\DOWN,Y^\UP) \iff X^\DOWN \subseteq Y^\DOWN \mbox{ \
and \ } X^\UP \subseteq Y^\UP. 
\end{equation*}

J.~Pomyka{\l }a and J.~A.~Pomyka{\l }a showed in \cite{PomPom88} that 
$\mathcal{RS} = (\textit{RS\/},\leq)$ is a Stone lattice. 
Later this result was improved by S.~D.~Comer \cite{Com91} by showing that in fact 
$\mathcal{RS}$ is a regular double Stone lattice. Note that a double Stone lattice 
$(L,\leq)$ with a pseudocomplement $^* \colon L \to L$ and a dual pseudocomplement
$^+ \colon L \to L$ is a regular double Stone lattice
if $x^* = y^*$ and $x^+ = y^+$ imply $x = y$ for all $x,y \in L$; see \cite{Schmitt76}. 

Finally, in \cite{GeWa92} M.~Gehrke and E.~Walker described the structure of 
$\mathcal{RS}$ precisely. They showed that $\mathcal{RS}$ is isomorphic to 
$\mathbf{2}^I \times \mathbf{3}^J$, where $\mathbf{2}$ and $\mathbf{3}$ are
the chains of two and three elements, $I$ is the set of singleton 
$E$-classes, $J$ is the set of non-singleton equivalence classes of $E$,
$\mathbf{2}^I$ is the pointwise ordered set of all mappings from $I$ to the
two-element chain, and $\mathbf{3}^J$ is the pointwise ordered set of all
maps from $J$ to the 3-element chain. Note that if each element of $U$ is 
indiscernible only with itself, then $E$ is the identity relation, all $E$-classes 
are singletons, and $\mathcal{RS}$ is isomorphic to $\mathbf{2}^U$. 
This may be interpreted so that rough sets really generalize ``classical sets''.

In the literature can be found numerous studies on rough sets that are determined 
by so-called \emph{information relations} reflecting distinguishability or indistinguishability 
of the elements of the universe of discourse; see \cite{Jarv07,KomPolSko98} for further references. 
For instance, E.~Or{\l}owska and Z.~Pawlak introduced in \cite{OrlPaw84} many-valued 
information systems in which each attribute attaches a set of values to objects. 
Therefore, in many-valued information systems, it is possible to express, 
for example, similarity, informational inclusion, diversity, and orthogonality  
in terms of information relations.

The idea now is that $R$ may be an arbitrary information relation, and rough lower and 
upper approximations are then defined in terms of $R$. 
This means that $x \in X^\UP$ if there is $y
\in X$ such that $x \,R\, y$, and $x \in X^\DOWN$ if $x \,R\, y$ implies $y
\in X$. Rough equality relation, the set of rough sets $\textit{RS\/}$, and
its partial order are defined as before. 
This kind of generalization is well justified since now it is possible to study
structures determined by other possible types of relation between objects, 
such as, for example, similarity or order.

It is known that if $R$ is reflexive and symmetric, then $\mathcal{RS}$ is
not always even a semilattice \cite{Jarv00}. Similarly, if $R$ is just
transitive, $\mathcal{RS}$ is not necessarily a semilattice \cite{Jarv04}.
However, if $R$ is symmetric and transitive, $\mathcal{RS}$ is a complete
double Stone lattice \cite{Jarv04}. Unfortunately, the structure of 
$\mathcal{RS}$ in case $R$ is a quasiorder, that is, $R$ is reflexive and
transitive, has been unknown. In this paper, we prove that if $R$ is a
quasiorder, then $\mathcal{RS}$ is a complete sublattice of $\wp(U) \times \wp(U)$.

This paper is structured as follows. In the next section, we give the
definition of rough sets determined by arbitrary relations, and recall some
of their well-known properties. We also present a decomposition theorem for
rough sets that are defined by a relation that is at least left-total. At the
end of the section, we recall the essential connection between quasiorders and Alexandrov 
topologies. Section~\ref{Sec:THREE} is devoted to our main result showing that 
for any quasiorder $R$, $\mathcal{RS}$ is a complete sublattice of 
$\wp(U) \times \wp(U)$. Note that this implies
directly that $\mathcal{RS}$ is completely distributive.
Then, we study the lattice structure of $\mathcal{RS}$ more carefully in 
Section~\ref{Sec:Complements}. We show that there can be defined three
different kinds of complementation operations. The completely join-irreducible
and completely meet-irreducible elements of $\mathcal{RS}$ are described in 
Section~\ref{Sec:JoinIrreducibles}. In Section~\ref{Sec:Stonean}, we characterize 
the case in which $\mathcal{RS}$ is a Stone lattice.

\section{Rough Set Approximations} \label{Sec:TWO}

We begin by defining the rough set approximations based on arbitrary binary
relations. Let $R$ be any binary relation on $U$. We denote for any $x \in U$, 
$R(x) = \{ y \in U \mid x \,R\, y \}$. For any subset $X \subseteq U$, the 
\emph{lower approximation} of $X$ is 
\begin{equation*}
X^\DOWN = \{x \in U \mid R(x) \subseteq X \} 
\end{equation*}
and the \emph{upper approximation} of $X$ is 
\begin{equation*}
X^\UP = \{x \in U \mid R(x) \cap X \neq \emptyset \} . 
\end{equation*}
Let $X^c$ denote the \emph{complement} $U \setminus X$ of $X$. Then, 
\begin{equation*}
X^{\UP c} = X^{c \DOWN} \mbox{ \ and \ }
X^{\DOWN c} = X^{c \UP}, 
\end{equation*}
that is, $^\DOWN$ and $^\UP$ are \emph{dual}. In addition, 
\begin{equation*}
\Big ( \bigcup \mathcal{H} \Big )^\UP = \bigcup \{ X^\UP \mid X \in \mathcal{H} \} 
\end{equation*}
and 
\begin{equation*}
\Big ( \bigcap \mathcal{H} \Big )^\DOWN = \bigcap \{ X^\DOWN \mid X \in 
\mathcal{H} \} 
\end{equation*}
for all $\mathcal{H} \subseteq \wp(U)$.
The last two equations imply that the maps $^\UP$ and $^\DOWN$ are order-preserving.

We assume that the reader is familiar with the notions of reflexive, 
symmetric, transitive, quasiorder, and equivalence relations. A
relation $R$ is \emph{left-total}, if for all $x \in U$, there exists $y \in U$
such that $x \inR{R} y$. Note that every reflexive relation is left-total.
In the literature left-total relations are also called  \emph{total}
or \emph{serial relations}, and quasiorders are named as \emph{preorders}. 

Below is listed how properties of the relation $R$ may be expressed 
in terms of approximations. Note that these well-known equivalences
are closely related to correspondence results between modal logic axiom
schemata and different types of Kripke frames; see \cite{Benthem84,BGO94},
for instance. For any binary relation $R$ on $U$, 
\begin{align*}
\mbox{ $R$ is left-total }     & \iff (\forall X \subseteq U)\, X^\DOWN \subseteq X^\UP, \\
\mbox{ $R$ is reflexive }  &\iff (\forall X \subseteq U) \, X \subseteq X^\UP, \\
\mbox{ $R$ is symmetric }  & \iff (\forall X \subseteq U)\, X \subseteq X^{\UP \DOWN}, \\
\mbox{ $R$ is transitive } & \iff (\forall X \subseteq U)\, X^{\UP \UP} \subseteq X^\UP.
\end{align*}
It should be noted that rough sets have close connections to modal, intuitionistic 
and many-valued logics \cite{PagChak08}.

Rough sets may now be defined as in case of equivalences. Let us denote for
any $X \subseteq U$, 
\begin{equation*}
\mathcal{A}(X) = (X^\DOWN,X^\UP), 
\end{equation*}
and call it the \emph{rough set} of $X$. Furthermore, we denote by 
\begin{equation*}
\textit{RS\/} = \{ \mathcal{A}(X) \mid X \subseteq U \} 
\end{equation*}
the set of all rough sets. The set $\textit{RS\/}$ can be ordered
coordinatewise by 
\begin{equation*}
(X^\DOWN,X^\UP) \leq (Y^\DOWN,Y^\UP) \iff X^\DOWN \subseteq Y^\DOWN 
\mbox{ \ and \ } X^\UP \subseteq Y^\UP, 
\end{equation*}
obtaining in this way a bounded partially ordered set
$\mathcal{RS} = (\textit{RS\/},\leq)$ with $\mathcal{A}(\emptyset) = (\emptyset^\DOWN,\emptyset)$ 
as the least element and $\mathcal{A}(U) = (U,U^\UP)$ as the greatest element.
A rough set $\mathcal{A}(X)$ is called an \emph{exact} element of $\textit{RS\/}$ if
$X^\DOWN = X = X^\UP$. 

Let us define the mapping:
\[
c \colon \textit{RS\/} \to \textit{RS\/}, \mathcal{A}(X) \mapsto \mathcal{A}(X^c).
\]
Since $c((X^\DOWN,X^\UP)) = (X^{\UP c}, X^{\DOWN c})$ for any $X \subseteq U$, the
mapping $c$ is well defined, and it is easy to see that the pair $(c,c)$ is an
order-reversing Galois connection on $\mathcal{RS}$. This implies directly the
following proposition.

\begin{proposition} \label{Prop:SelfDual}
The partially ordered set $\mathcal{RS}$ is self-dual, that is, $\mathcal{RS}$ is
order-isomorphic to its dual $\mathcal{RS}^\mathrm{op}$.
\end{proposition}

For any binary relation $R$ on $U$, a set $C$ is called a \emph{connected component} of $R$, 
if $C$ is an equivalence class of the smallest equivalence relation containing $R$.
Let us denote by $\mathfrak{Co}$ the set of all connected components. 
Clearly, for any connected component $C \in \mathfrak{Co}$ and $x \in U$,
$R(x)\cap C \neq \emptyset $ implies $x \in C$, and $x \in C$ implies $R(x)\subseteq C$.
Hence, $C^\UP \subseteq C \subseteq C^\DOWN$. 

Let $R$ be a left-total relation. Because $X^\DOWN \subseteq X^\UP$ for every $X \subseteq U$,
any connected component $C \in \mathfrak{Co}$ of $R$ satisfies
$C^\DOWN = C^\UP = C$ and thus $\mathcal{A}(C)=(C,C)$ is an exact element of $\textit{RS\/}$. 
Additionally, we denote for each $C \in \mathfrak{Co}$ by $\textit{RS\/}(C)$ the set of rough 
sets on the component  $C$ determined by the restriction of $R$ to $C$. 
The corresponding ordered set is denoted by $\mathcal{RS}(C)$.

Next we present a decomposition theorem for rough sets determined by left-total
relations. First, we prove the following lemma.

\begin{lemma} \label{Lem:Separation}
If $R$ is a left-total relation on $U$, then the following assertions hold.
\begin{enumerate}[\rm (i)]
\item If $\{C_i \mid i \in I\} \subseteq \mathfrak{Co}$ is a subset of connected 
components of $R$,  then for any family $\{ (X_i^\DOWN,X_i^\UP) \in \textit{RS\/}(C_i) \mid i \in I \}$,
the pair $\left( \bigcup_{i \in I} X_i^\DOWN,\bigcup_{i \in I} X_i^\UP \right)$ is a rough set on $U$.

\item If $(X^\DOWN,X^\UP)$ is a rough set on $U$, then the pair $(X^\DOWN \cap C,X^\UP\cap C)$ is a rough 
set on $C$ for any connected component $C \in \mathfrak{Co}$.
\end{enumerate}
\end{lemma}

\begin{proof}
(i) \ Clearly, $\left( \bigcup_{i \in I} X_i \right ) ^\UP = \bigcup_{i \in I} X_i^\UP$.
We will show that also $\left( \bigcup_{i \in I} X_i \right )^\DOWN = \bigcup_{i \in I} X_i^\DOWN$.
Since $X_i^\DOWN \subseteq \left( \bigcup_{i \in I}X_i \right)^\DOWN$ for all $i \in I$, we have
$\bigcup_{i \in I} X_i^\DOWN \subseteq \left ( \bigcup_{i \in I} X_i \right)^\DOWN$. 
On the other hand, let $x \in \left( \bigcup_{i \in I}X_i \right )^\DOWN$. 
Because $R(x) \neq \emptyset$ and $R(x)\subseteq \bigcup_{i \in I}X_i \subseteq \bigcup_{i \in I} C_i$, 
there exist $k \in I$ and $y \in C_k$ such that $x \inR{R} y$. 
Hence, $x\in C_k$ and $R(x) \subseteq C_k$. This implies
\[
R(x) \subseteq C_k \cap \Big ( \bigcup_{i \in I} X_i \Big) = \bigcup_{i \in I} \Big ( C_k \cap X_i \Big ) = X_k,
\]
because $X_k \subseteq C_k$ and $X_i \cap C_k = \emptyset$ for all $i \in I \setminus \{ k \}$.
Thus,  we obtain $x \in X_i^\DOWN \subseteq \bigcup_{i \in I} X_i^\DOWN$, which gives
$\left( \bigcup_{i \in I} X_i \right)^\DOWN = \bigcup_{i \in I} X_i^\DOWN$. 
Therefore,
$\mathcal{A} \left( \bigcup_{i \in I} X_i\right) = 
\left( \bigcup_{i \in I}X_i^\DOWN,\bigcup_{i \in I} X_i^\UP \right)$ is a rough set on $U$.
\medskip%

(ii) Let $C \in \mathfrak{Co}$. 
We prove that $(X\cap C)^\DOWN=X^\DOWN \cap C$ and $(X\cap C)^\UP = X^\UP\cap C$.
It is easy to see that
\[ 
 (X \cap C)^\DOWN = X^\DOWN \cap C^\DOWN = X^\DOWN \cap C.
\]
Furthermore, $\left ( X \cap C \right )^\UP \subseteq X^\UP$ and 
$\left( X \cap C \right )^\UP \subseteq C^\UP = C$ imply
$(X \cap C)^\UP \subseteq X^\UP \cap C$.
For the converse, suppose that $x\in X^\UP \cap C$. 
Then, $R(x) \cap X  \neq \emptyset$ and $R(x) \subseteq C$ imply 
$R(x) \cap (X\cap C) = (R(x) \cap C) \cap X = R(x) \cap X \ne \emptyset$, 
that is, $x \in (X \cap C)^\UP$. This proves $(X \cap C)^\UP = X^\UP \cap C$. 
\end{proof}

\begin{corollary} \label{Cor:Exact}
If $R$ is a left-total relation, then for any subset 
$\mathcal{H} \subseteq \mathfrak{Co}$ of the connected components of $R$, 
the rough set $\mathcal{A}\left( \bigcup \mathcal{H} \right )$ 
is an exact element of $\textit{RS\/}$.
\end{corollary}

\begin{proof}
Let $\mathcal{H} \subseteq \mathfrak{Co}$.
By the proof of Lemma~\ref{Lem:Separation}(i),
\[
\left ( \bigcup \mathcal{H}  \right )^\DOWN =
\bigcup_{C \in \mathcal{H}} C^\DOWN = 
\bigcup \mathcal{H} 
= \bigcup_{C \in \mathcal{H}} C^\UP = 
\left ( \bigcup \mathcal{H} \right )^\UP.
\]
\end{proof}

For any index set $I$, let $\mathcal{P} = \prod_{i \in I} \mathcal{P}_i$ be 
the  \emph{Cartesian product} of the partially ordered sets $\mathcal{P}_i = (P_i,\leq_i)$,
and let $x_i$ denote the $i$-th coordinate of an element 
$x \in P= \prod_{i \in I}P_i$. We will also write $x = (x_i)_{i \in I}$. 
Recall that the partial order $\leq$ of $\mathcal{P}$ is
defined coordinatewise (see e.g \cite{Trott92}), that is,  for any $x,y\in P$, we have 
$x\leq y$ if and only if $x_i \leq_i y_i$ for all $i \in I$.

\begin{theorem}\label{Thm:Representation} If $R$ is a left-total relation on $U$, then
$\mathcal{RS}$ is order-isomorphic to $\prod_{C \in \mathfrak{Co}} \mathcal{RS}(C)$.
\end{theorem}

\begin{proof} Assume that $\mathfrak{Co}  = \{ C_i \mid i \in I\} $. 
Consider the maps 
\[ \Phi\colon \textit{RS\/} \to \prod_{i \in I} \textit{RS\/}(C_{i})
\mbox{ \quad and \quad }
\Psi \colon \prod_{i \in I} \textit{RS\/}(C_{i})\to \textit{RS\/}
\] 
defined by
\[
\Phi(\mathcal{A}(X)) = ((X^\DOWN\cap C_{i},X^{\UP}\cap C_{i}))_{i\in I}
\]
for any $\mathcal{A}(X) = (X^{\DOWN},X^\UP) \in \textit{RS\/}$, and
\[
\Psi(((X_{i}^\DOWN,X_{i}^\UP))_{i\in I}) =
\Big ( \bigcup_{i \in I }X_{i}^\DOWN, \bigcup_{i \in I}X_{i}^\UP\Big)
\] 
for any $((X_{i}^\DOWN,X_{i}^\UP))_{i\in I}\in \prod_{i \in I}\textit{RS\/}(C_{i})$.

\medskip 

In view of Lemma~\ref{Lem:Separation}, the maps $\Phi$ and $\Psi$ are well-defined, 
and it is easy to check that both $\Phi$ and $\Psi$ are order-preserving.
Hence, to prove that $\Phi$ and $\Psi$ are order-isomorphisms, it is enough to show 
that they are mutually inverse maps.

For any $\mathcal{A}(X) = (X^\DOWN,X^\UP) \in \textit{RS\/}$, we obtain:
\begin{align*}
\Psi\left( \Phi(\mathcal{A}(X))\right) & =\Psi (((X^\DOWN\cap C_{i},X^\UP\cap C_{i}))_{i\in I}) \\
      & = \Big ( \bigcup_{i \in I} (X^\DOWN \cap C_{i}), \bigcup_{i \in I} (X^\UP\cap C_{i}) \Big ) \\
      & = \Big ( X^\DOWN \cap \big( \bigcup_{i \in I}C_{i} \big ), X^\UP\cap \big ( \bigcup_{i \in I}C_{i} \big ) \Big ) \\
      & = (X^\DOWN,X^\UP) \\
      & = \mathcal{A}(X).
\end{align*}
Furthermore, for any $((X_{i}^\DOWN,X_{i}^\UP))_{i \in I} \in \prod_{i \in I} \textit{RS\/}(C_{i})$, we have:
\begin{align*}
\Phi (\Psi(((X_{i}^\DOWN,X_{i}^\UP))_{i \in I})) 
   & =  \Phi \Big ( \big ( \bigcup_{j \in I} X_{j}^\DOWN,\bigcup_{j \in J}X_{j}^\UP \big ) \Big ) \\
   & =  \Big ( \Big ( \big ( \bigcup_{j \in I} X_{j}^\DOWN \big ) \cap C_{i}, 
                      \big ( \bigcup_{j\in I} X_{j}^\UP \big ) \cap C_{i} \Big ) \Big)_{i\in I} \, .
\end{align*}
Because $(X_j^\DOWN,X_j^\UP) \in \textit{RS\/}(C_j)$ for every $j \in I$, 
$X_{j}^\DOWN$ and  $X_{j}^\UP$ are subsets of $C_{j}$. 
Additionally,  for any $i,j \in I$ such that $i \neq j$, $C_{i}\cap C_{j} = \emptyset$.
These facts imply that
\[
\Big( \bigcup_{j \in I}X_{j}^\DOWN \Big) \cap C_{i} 
 = \bigcup_{j \in I} \left ( X_{j}^\DOWN \cap C_{i} \right) = X_{i}^\DOWN 
\]
and 
\[
\Big ( \bigcup_{j \in I}X_{j}^\UP \Big ) \cap C_{i} 
= \bigcup_{j \in I} \left ( X_{j}^\UP \cap C_{i} \right ) = X_{i}^\UP,
\]
for each $i\in I$. Thus, we obtain 
$\Phi(\Psi(((X_{i}^\DOWN,X_{i}^{\UP}))_{i\in I}))=((X_{i}^\DOWN,X_{i}^\UP))_{i\in I}$.

In view of the above equalities, the maps $\Phi $ and $\Psi $ are order-preserving
inverse mappings of each other and hence they are order-isomorphisms. 
So, $\mathcal{RS}$ and $\prod_{C \in \mathfrak{Co}}\mathcal{RS} (C)$ are order-isomorphic.
\end{proof}

We may also determine rough set approximations in terms of the inverse $R^{-1}$ of $R$, 
that is, 
\begin{equation*}
X^\Down = \{x \in U \mid R^{-1}(x) \subseteq X \} 
\end{equation*}
and 
\begin{equation*}
X^\Up = \{x \in U \mid R^{-1}(x) \cap X \neq \emptyset \}. 
\end{equation*}
Interestingly, the pairs $(^\UP,^\Down \! )$ and $(^\Up,^\DOWN \! )$ are
order-preserving Galois connections on $\wp(U)$. The end of this section is
devoted to approximations determined by quasiorders. First, we recall the
notion of Alexandrov topologies that is closely connected to quasiorders -- 
for further details see \cite{Alex37,Birk37,ErneRein95}, for example.

An \emph{Alexandrov topology} is a topology $\mathcal{T}$ that
contains also all arbitrary intersections of its members.
Let $\mathcal{T}$ be an Alexandrov topology  $\mathcal{T}$ on $U$.
Then, for each $X \subseteq U$, there exists the \emph{smallest neighbourhood}
\[
N_\mathcal{T}(X) = \bigcap \{ Y \in \mathcal{T} \mid X \subseteq Y \}.
\]
In particular, the smallest neighbourhood of a point $x \in U$
is denoted by $N_\mathcal{T}(x)$. The family 
\begin{equation*}
\mathcal{B_T} = \{ N_\mathcal{T}(x) \mid x \in U\} 
\end{equation*}
is the \emph{smallest base} of the Alexandrov topology $\mathcal{T}$. 
This means that every member $X$ of $\mathcal{T}$ can be expressed
as a union of some (or none) elements of $\mathcal{B_T}$, that is,
$X = \bigcup \{ N_\mathcal{T}(x) \mid x \in X \}$. In addition, $\mathcal{B}_\mathcal{T}$
is smallest such set.

There is a close connection between quasiorders and Alexandrov topologies.
This correspondence will turn very useful in Section~\ref{Sec:JoinIrreducibles},
where we will study the completely join-irreducible elements of $\mathcal{RS}$.
Let $R$ be a quasiorder on a set $U$. We may now define an Alexandrov topology 
${\mathcal T}_R$ on $U$ consisting of all ``upward-closed'' subsets of $U$ with respect to the
relation $R$, that is, 
\[
\mathcal{T}_R = \{ A \subseteq U \mid (\forall x,y \in U) \; x \in A \ \ \& \ \ x \inR{R} y \Longrightarrow y \in A \} \\
\]
On the other hand, the set $R(x)$ is the smallest neighbourhood
of the point $x$ in the Alexandrov topology $\mathcal{T}_R$ and
clearly $y \in R(x)$ if  and only if $x \inR{R} y$. This hints how we
may also determine quasiorders by means of Alexandrov topologies. If
$\mathcal{T}$ is an Alexandrov topology on $U$, then we define a
quasiorder  $R_\mathcal{T}$ on $U$ by setting
\[ x \inR{R_\mathcal{T}} y \iff y \in  N_\mathcal{T}(x). \]
The correspondences $R\ \mapsto \mathcal{T}_R$ and $\mathcal{T} \mapsto \ R_\mathcal{T}$
are one-to-one.  It is also well known that the categories of quasiordered sets and 
Alexandrov spaces are isomorphic, as discussed in \cite{Erne91}, for example.

For a quasiorder $R$, the rough approximations satisfy for all $X \subseteq U$: 
\begin{equation*}
X^{\UP \Down} = X^\UP, \ X^{\Up \DOWN} = X^\Up, \ X^{\DOWN \Up} = X^\DOWN, \ X^{\Down \UP} = X^\Down. 
\end{equation*}
These approximations determine two Alexandrov topologies on $U$: 
\begin{equation*}
\mathcal{T}^\UP = \{ X^\UP \mid X \subseteq U\} = \{ X^\Down \mid X
\subseteq U\} 
\end{equation*}
and 
\begin{equation*}
\mathcal{T}^\DOWN = \{ X^\DOWN \mid X \subseteq U\} = \{ X^\Up \mid X
\subseteq U\}. 
\end{equation*}
Note that $\mathcal{T}^\DOWN$ is the same as $\mathcal{T}_R$ above.
Clearly, these topologies are dual, that is, for all $X \subseteq U$, 
\begin{equation*}
X \in \mathcal{T}^\UP \iff X^c \in \mathcal{T}^\DOWN 
\end{equation*}
For the Alexandrov topology $\mathcal{T}^\UP$:

\begin{enumerate}[(i)]

\item $^\UP \colon \wp(U) \to \wp(U)$ is the smallest neighbourhood operator.

\item $^\Up \colon \wp(U) \to \wp(U)$ is the closure operator. Note that the
family of closed sets for the topology $\mathcal{T}^\UP$ is $\mathcal{T}^\DOWN$.

\item $^\Down \colon \wp(U) \to \wp(U)$ is the interior operator, that is,
it maps each set to the greatest open set contained into the set in question.

\item The set $\{ \, \{x\}^\UP \mid x \in U\} = \{ R^{-1}(x) \mid x \in U\}$
is the smallest base.
\end{enumerate}
Similarly, for the topology $\mathcal{T}^\DOWN$:

\begin{enumerate}[(i)]

\item $^\Up \colon \wp(U) \to \wp(U)$ is the smallest neighbourhood operator.

\item $^\UP \colon \wp(U) \to \wp(U)$ is the closure operator.

\item $^\DOWN \colon \wp(U) \to \wp(U)$ is the interior operator.

\item The set $\{ \, \{x\}^\Up \mid x \in U\} = \{ R(x) \mid x \in U\}$ is
the  smallest base.
\end{enumerate}

\section{Lattices of Rough Sets Determined by Quasiorders} \label{Sec:THREE}

In this section, we prove that the quasiorder-based rough sets form a complete lattice.

We start by considering cofinal sets. Let $R$ be a transitive relation
on a non-empty set $U$. A \emph{successor} of $x \in U$ is an element $y \in U$
such that $x \inR{R} y$. Let $X \subseteq Y \subseteq U$. Then,
$X$ is \emph{cofinal in} $Y$ if each $x \in Y$ has a successor in $X$.
By using the notation introduced in Section~\ref{Sec:TWO}, the set
of successors of $x$ is simply $R(x)$. Additionally, $X$ is
cofinal in $Y$ if and only if $R(x) \cap X \neq \emptyset$ for
all $x \in Y$, which is equivalent to $Y \subseteq X^\UP$. 
Since $X \subseteq Y$, this actually means that $X$ is cofinal in $Y$
if and only if $X^\UP = Y^\UP$. We also say that a set is \emph{cofinal},
if it is cofinal in $U$.

In the proof of our main result, we will use the following theorem
for transitive relations on $U$ by A.~H.~Stone.

\begin{theorem}[Theorem~1 of \cite{Stone68}] \label{Thm:Stone}
A necessary and sufficient condition that the set $U$ has a partition into 
$k$ cofinal subsets, is that each element of $U$ has at least $k$ successors.
\end{theorem}

Let $R$ be a quasiorder on $U$. Then for all $a,b \in U$,
\[
a \inR{R} b \iff b \in R(a) \iff R(b) \subseteq R(a),
\]
and $|R(a)| \geq 1$. Recall from Section~\ref{Sec:TWO} that since $R$ is a quasiorder, 
$^\UP$ is a closure operator and $^\DOWN$ is an interior operator. 
Thus, for any $X \subseteq U$,
\[
X^\DOWN \subseteq X \subseteq X^\UP, \quad
X^{\DOWN \DOWN} = X^\DOWN, \mbox{ \quad and \quad}
X^{\UP \UP} = X^\UP.
\]
In addition, $X \subseteq Y$ implies $X^\DOWN \subseteq Y^\DOWN$ and $X^\UP \subseteq Y^\UP$
for any $X,Y \subseteq U$. These properties are needed in the proof of our next theorem.

As we already mentioned, J.~Pomyka{\l }a and J.~A.~Pomyka{\l }a showed in \cite{PomPom88} that 
for equivalence relations, $\mathcal{RS}$ is a Stone lattice. In their proof they 
used Zermelo's Axiom of Choice. Note that the proof of Theorem~\ref{Thm:Stone} above
also requires Axiom of Choice.

\begin{theorem} \label{Thm:Main}
If $R$ is a quasiorder on a non-empty set $U$, then
$\mathcal{RS}$ is a complete sublattice of $\wp(U) \times \wp(U)$.
\end{theorem}

\begin{proof} To prove that $\mathcal{RS}$ is a complete sublattice of 
$\wp(U) \times \wp(U)$, it suffices to show that for any 
family  $\{ \mathcal{A}(X_{i}) \mid i \in I \} = 
\{ \left( X_{i}^\DOWN,X_{i}^\UP\right) \mid i\in I \}$ of rough sets, the pairs 
$\left( \bigcap_{i \in I} X_{i}^\DOWN, \bigcap_{i \in I} X_{i}^{\UP}\right)$ 
and $\left( \bigcup_{i \in I} X_{i}^\DOWN, \bigcup_{i \in I} X_{i}^{\UP}\right)$ 
also are rough sets.

\medskip

(1) First, we construct a set $W \subseteq U$ that satisfies
$W^\DOWN = \bigcap_{i \in I} X_{i}^\DOWN$ and $W^\UP= \bigcap_{i \in I} X_{i}^\UP$.
Let us consider the set 
\[ 
Z = \bigcap_{i \in I}  X_{i}^\UP \setminus \Big ( \bigcap_{i\in I} X_{i} \Big )^\UP,
\]
and observe that for any $a\in Z$, we have $\left\vert R(a)\right\vert \geq 2$.
Indeed, by the definition of $Z$, $a\in Z$ means that  $R(a)\cap X_{i} \neq \emptyset $ for all $i\in I$
and $R(a) \cap\left( \bigcap_{i \in I} X_{i} \right) = \emptyset$.
If $R(a)$ has the form $R(a)=\{a\}$, then $R(a)\cap X_{i} \neq \emptyset$
implies  $a\in X_{i}$ for each $i\in I$, that is,  $a \in \bigcap_{i \in } X_{i}$, a
contradiction.

Let us first consider the case $Z^\DOWN \neq \emptyset$. Note that each successor
of any $a \in Z^\DOWN$ is also in $Z^\DOWN$, because if $b$ is a successor of $a$,
then $a \inR{R} b$ implies $R(b) \subseteq R(a) \subseteq Z$, that is, $b \in Z^\DOWN$.
This then means that each $a \in Z^\DOWN$ has at least two successors in $Z^\DOWN$,
and we may apply Theorem~\ref{Thm:Stone} to the set $Z^\DOWN$  with the
relation $R' = R \cap (Z^\DOWN \times Z^\DOWN)$. So, there exist two disjoint
sets $A,B \subseteq Z^\DOWN$ that are cofinal in the quasiordered set $(Z^\DOWN,R')$.
But since $R' \subseteq R$, $A$ and $B$ are cofinal in $Z^\DOWN$ with respect
to the original relation $R$. This gives $Z^\DOWN \subseteq A^\UP$ and
$Z^\DOWN \subseteq B^\UP$. In the case $Z^\DOWN = \emptyset$, we set $A = B = \emptyset$.

Let us define the set
\[
W = \Big ( \bigcap_{ i \in I} X_{i} \Big ) \cup  (Z \setminus A  ).
\]
We will show that $W^\DOWN = \bigcap_{i \in I} X_{i}^\DOWN$ and $W^\UP = \bigcap_{i \in I} X_{i}^\UP$.

Since $\bigcap_{i \in I} X_{i} \subseteq W$,  
$\bigcap_{i \in I} X_{i}^\DOWN = \left( \bigcap_{i \in I} X_{i}\right) ^\DOWN \subseteq W^\DOWN$. 
In order to prove the converse inclusion  $W^\DOWN \subseteq \bigcap_{i \in I} X_{i}^\DOWN$, 
suppose that $a\in W^\DOWN$.  Then, $R(a)\subseteq W = \left( \bigcap_{i \in I} X_{i}\right) \cup (Z \setminus A)$, 
and next we will prove that necessarily $R(a)\cap (Z \setminus A) = \emptyset$.

Indeed, if $R(a) \cap (Z \setminus A) \neq \emptyset$, then there is 
$b \in R(a) \cap (Z \setminus A)$. So, $R(b) \subseteq R(a) \subseteq W$, and
$b \in Z$ implies $R(b) \cap \left (   \bigcap_{i \in I} X_{i} \right ) = \emptyset$.
Therefore, we must have $R(b) \subseteq Z \setminus A$. In the case $Z^\DOWN = \emptyset$,
this implies a contradiction. If $Z^\DOWN \neq \emptyset$, then from the fact that
$A$ is cofinal in $Z^\DOWN$, we get $b \in Z^\DOWN \subseteq A^\UP$. On the other hand,
$R(b) \subseteq (Z \setminus A)$ implies $R(b) \cap A = \emptyset$ contradicting
$b \in A^\UP$. Hence, we have $R(a) \cap (Z \setminus A) = \emptyset$.

The facts $R(a)\cap (Z \setminus A) = \emptyset$  and  $R(a)\subseteq W$ imply $R(a)\subseteq \bigcap_{i \in I} X_{i}$, 
that is, $a \in \left( \bigcap_{i\in I} X_i \right )^\DOWN = \bigcap_{i\in I} X_{i}^\DOWN$. 
Therefore, we obtain $W^\DOWN = \bigcap_{i\in I} X_{i}^\DOWN$.

To conclude part (1), let us show the equality $W^\UP = \bigcap_{i \in I} X_i^\UP$.
Obviously, $\bigcap_{i\in I}X_i \subseteq \bigcap_{i\in I} X_i^\UP$ and 
$(Z \setminus A) \subseteq \bigcap_ {i \in I} X_i^\UP$ imply $W \subseteq \bigcap_{i\in I} X_{i}^\UP$. 
Since Alexandrov topologies are closed also with respect to arbitrary intersections,
$\bigcap_{i\in I} X_{i}^\UP$ belongs to $\mathcal{T}^\UP$, and therefore
$W^\UP \subseteq \left ( \bigcap_{i\in I} X_{i}^\UP \right )^\UP = \bigcap_{i\in I} X_{i}^\UP$.

Conversely, we prove $\bigcap_{i\in I} X_{i}^\UP \subseteq W^\UP$ by showing first that 
$Z \subseteq (Z \setminus A)^\UP$. For that, take any $z \in Z$. Clearly, we may now assume $z \in A$,
because otherwise there is nothing left to prove. Because $A$ is cofinal in $Z^\DOWN$,
we have $A \subseteq Z^\DOWN \subseteq B^\UP$. In addition, $B \subseteq (Z \setminus A)$
implies $B^\UP \subseteq (Z \setminus A)^\UP$. Thus, $z \in A \subseteq (Z \setminus A)^\UP$
proving $Z \subseteq (Z \setminus A)^\UP$. Therefore, we may write:

\begin{align*}
\bigcap_{i\in I} X_{i}^\UP 
  & = \Big ( \bigcap_{i\in I} X_i \Big )^\UP \cup \Big( \bigcap_{i\in I} X_{i}^\UP 
       \setminus \Big ( \bigcap_{i\in I} X_{i} \Big )^\UP \Big ) \\
  & = \Big ( \bigcap_{i\in I} X_{i} \Big )^\UP \cup Z \\
  & \subseteq \Big ( \bigcap_{i\in I} X_{i} \Big )^\UP \cup (Z \setminus A)^\UP \\
  & \subseteq W^\UP.
\end{align*}
This proves $W^\UP = \bigcap_{i \in I} X_{i}^\UP$. 
Hence, $(W^\DOWN,W^\UP) = \left( \bigcap_{i\in I} X_i^\DOWN,\bigcap_{i\in I} X_{i}^\UP \right)$ is a rough set.

\medskip

(2) We will prove that $\left( \bigcup_{i\in I} X_{i}^\DOWN,\bigcup_{i\in I} X_{i}^\UP \right)$ 
is a rough set. This is done by constructing a set $V\subseteq U$ such that
$V^\DOWN = \bigcup_{i\in I} X_{i}^\DOWN$ and $V^\UP = \bigcup_{i\in I}X_{i}^\UP$.
Let us first consider the set 
\[
S = \Big ( \bigcup_{i\in I} X_{i} \Big )^\UP \setminus \Big ( \bigcup_ {i\in I} X_{i} ^\DOWN \Big) 
\]
and observe that for each $b\in S$, $| R(b) | \geq 2$. Indeed, if we suppose that 
$R(b)$ has only one element, that is, $R(b)=\{b\}$, then 
$b \in \left( \bigcup_{i\in I} X_{i} \right) ^\UP = \bigcup_{i\in I} X_{i}^\UP$ 
implies $R(b) \cap X_k \neq \emptyset$ for some $k \in I$, that is $b\in X_k$. 
However, in this case $\{b\} = R(b)\subseteq X_k$ implies $b\in$ $X_k^\DOWN$, a contradiction.

Let us first assume that $S^\DOWN \neq \emptyset$. As in case (1), we note that
every $a \in S^\DOWN$ has at least two successors in $S^\DOWN$ and we may deduce that
there exists two disjoint cofinal subsets $A$ and $B$ of $S^\DOWN$. This means
that $S^\DOWN \subseteq A^\UP$ and $S^\DOWN \subseteq B^\UP$. Furthermore, if $S^\DOWN = \emptyset$,
we set $A = B = \emptyset$.

Let us define the sets $H$ and $V$ by
\begin{align*}
H & = \Big \{ a \in S \mid R(a) \nsubseteq \Big ( \bigcup_{i\in I} X_{i} \Big )^\UP \Big \}, \\ 
V     & = \Big( \bigcup_{i\in I} X_i^\DOWN \Big ) \cup H \cup A.
\end{align*}
Observe that $V \subseteq \left( \bigcup_{i\in I} X_{i} \right )^\UP$. This is because 
$\bigcup_{i\in I} X_{i}^\DOWN \subseteq \bigcup_ {i\in I} X_{i}\subseteq \left( \bigcup_ {i\in I} X_{i} \right)^\UP$
and $A, H \subseteq S \subseteq \left( \bigcup_{i\in I} X_{i}\right) ^\UP$ by definition. 

\medskip

Now, we will prove that 
\[ 
V^\DOWN = \bigcup_{i\in I} X_{i}^\DOWN \mbox{ \quad and \quad } V^\UP = \bigcup_{i\in I} X_{i}^\UP.
\]
Indeed, we have $X_{i}^\DOWN \subseteq V$ for all $i\in I$ by the definition of $V$. 
Therefore, for all $i \in I$, $X_{i}^{\DOWN} = X_{i}^{\DOWN\DOWN} \subseteq V^{\DOWN }$, which gives 
\[ \bigcup_{i\in I} X_{i}^\DOWN \subseteq V^\DOWN.\] To show the converse, 
suppose that $a\in V^\DOWN$. Then, 
\[
 a \in R(a) \subseteq V = \Big ( \bigcup_{i\in I} X_{i}^\DOWN \Big ) \cup H \cup A.
\]
Observe that $a \in H$ is not possible, since
$R(a) \subseteq V \subseteq \left (\bigcup_{i \in I} X_i \right )^\UP$.
Next we will show that $a \in A$ is also excluded.

Indeed, if $a \in A$, then necessarily $S^\DOWN \neq \emptyset$. The inclusions
$A \subseteq S^\DOWN \subseteq B^\UP$ imply $a \in B^\UP$.
This means that there exists an element $b \in B$ with $a \inR{R} b$ and
$b \in R(a) \subseteq \left ( \bigcup_{i \in I} X_i^\DOWN \right ) \cup H \cup A$.
Option  $b \in A$ is not possible, because $A \cap B = \emptyset$.
Also $b \in H$ is impossible, because 
$R(b) \subseteq R(a) \subseteq \left (\bigcup_{i \in I} X_i \right )^\UP$.
Finally, the remaining case $b \in  \bigcup_{i \in I} X_i^\DOWN$ contradicts with
$b \in B \subseteq S = \left ( \bigcup_{i\in I} X_{i} \right )^\UP \setminus \left ( \bigcup_ {i\in I} X_{i} ^\DOWN \right)$. 

Because we showed that $a \in H$ and $a \in A$ are not possible, we have that $a \in \bigcup_{i\in I} X_{i}^\DOWN$.
Therefore, the inclusion  $V^\DOWN  \subseteq \bigcup_{i\in I} X_{i}^\DOWN$ is verified and
we may write $V^{\DOWN } = \bigcup_{i\in I} X_{i}^\DOWN$.

\medskip

Let us prove now the equality $V^{\UP } = \bigcup_{i\in I} X_{i}^\UP$.
Because $V\subseteq \left( \bigcup_{i\in I} X_{i} \right)^\UP$, we have
\[
V^\UP \subseteq \Big ( \bigcup_{i\in I} X_{i} \Big ) ^{\UP \UP } =
\Big ( \bigcup_{i\in I} X_{i} \Big )^\UP =
\bigcup_{i\in I} X_{i}^\UP .
\]
To prove the reverse inclusion, suppose that $a\in \bigcup_{i\in I} X_{i}^{\UP }$. 
Then, there exists $k \in I$ such that $a \in X_k^\UP$, and this implies 
$R(a) \cap X_k \neq \emptyset$. If $R(a)\cap \left( \bigcup_{i\in I} X_{i}^\DOWN \right) 
\neq \emptyset $ holds, then $ a \in \left ( \bigcup_{i\in I} X_{i}^\DOWN \right )^\UP
\subseteq V^{\UP}$, and the proof is completed. Therefore, we now assume 
$R(a) \cap \left( \bigcup_{i\in I} X_{i}^\DOWN \right) =\emptyset$, that is,
\[
a \in \Big ( \bigcup_{i\in I} X_{i} \Big)^\UP \setminus \Big ( \bigcup_{i\in I} X_{i}^\DOWN \Big ) = S.
\]
If $R(a)\nsubseteq \left ( \bigcup_{i\in I} X_{i} \right)^\UP$, 
then $a\in H \subseteq V \subseteq V^\UP$ and the proof is again completed.
Hence, we restrict ourselves to the case $R(a)\subseteq \left ( \bigcup_{i\in I} X_{i}\right )^\UP$.

However, in this case we get  $R(a)\subseteq S$, because we have 
$R(a) \cap \left( \bigcup_{i\in I} X_{i}^\DOWN \right) =\emptyset$
by our hypothesis. Hence, we get $a \in S^\DOWN$ and therefore $S^\DOWN \neq \emptyset$.
Then, $S^\DOWN \subseteq A^\UP$, because $A$ is cofinal in $S^\DOWN$.
Since $A \subseteq V$, we conclude that $a \in A^\UP \subseteq V^\UP$.
This implies
\[
\bigcup_{i\in I} X_{i}^\UP = V^\UP.
\]
Hence, $\left( \bigcup_{i\in I} X_{i}^\DOWN, \bigcup_{i\in I} X_{i}^\UP \right)$ is a rough set, 
which completes the proof.
\end{proof}

The next corollary describes the meets and the joins in the complete lattice $\mathcal{RS}$.

\begin{corollary} \label{Cor:MeetsJoins}
If $R$ is a quasiorder on a non-empty set $U$, then
$\mathcal{RS}$ is a completely distributive complete lattice such that 
\[
 \bigwedge_{i\in I} \mathcal{A}(X_{i}) = \Big ( \bigcap_{i\in I} X_{i}^\DOWN, \bigcap_{i\in I} X_{i}^\UP \Big )
\mbox{ \quad and \quad }
 \bigvee_{i\in I} \mathcal{A}(X_{i}) = \Big ( \bigcup_{i\in I} X_{i}^\DOWN, \bigcup_{i\in I} X_{i}^\UP \Big )
\]
for all $\{ \mathcal{A}(X_{i}) \mid i \in I \} \subseteq \textit{RS\/}$.
\end{corollary}

\begin{proof} As $\mathcal{RS}$ is a complete sublattice of the completely distributive 
lattice $\wp(U)\times \wp(U)$, it is always completely distributive. Because the meet
and the join of $\{ \mathcal{A}(X_i) \mid i \in I\}$ in $\mathcal{RS}$ coincide with their 
meet and join in the lattice $\wp(U)\times \wp(U)$, we obtain the required formulas.
\end{proof}

\section{Complementation in the Lattice of Rough Sets}  \label{Sec:Complements}

In the previous section, we showed that for any quasiorder $R$, $\mathcal{RS}$ is a
completely distributive complete lattice. In this section, we describe
different types of complementation operations in $\mathcal{RS}$. 

Let us first consider the mapping
\[
c \colon \textit{RS\/} \to \textit{RS\/}, \, \mathcal{A}(X) \mapsto \mathcal{A}(X^c)
\]
introduced already in Section~\ref{Sec:TWO}. Now, for all $\alpha,\beta \in \textit{RS\/}$, we have
\begin{eqnarray*}
c(\alpha \vee \beta) & = & c(\alpha) \wedge c(\beta) \\
c(\alpha \wedge \beta) & = & c(\alpha) \vee c(\beta) \\
c(c(\alpha)) & = & \alpha
\end{eqnarray*}
This means that $c$ is so-called \emph{de~Morgan operation} on the lattice $\mathcal{RS}$
(for the notion, see \cite{BlythVarlet94}, for instance). 
Note that $\alpha \vee c(\alpha) = (U,U)$ and $\alpha \wedge c(\alpha) =
(\emptyset,\emptyset)$ do not generally hold. However, 
$X^\DOWN \cap X^{c \DOWN} = (X \cap X^c)^\DOWN = \emptyset^\DOWN = \emptyset$ and 
$X^\UP \cup X^{c \UP} = (X \cup X^c)^\UP = U^\UP = U$ for any $X \subseteq U$.

Let $(L,\leq)$ be a lattice with a least element $0$. An element $x^*$ is
a \emph{pseudocomplement} of $x \in L$, if $x \wedge x^* = 0$ and for all $a
\in L$, $x \wedge a = 0$ implies $a \leq x^*$. An element can have at most
one pseudocomplement. A lattice is \emph{pseudocomplemented} if each element
has a pseudocomplement.

Since any completely distributive complete lattice is both pseudocomplemented and
dually pseudocomplemented, we may write the following result by Corollary~\ref{Cor:MeetsJoins}.

\begin{proposition} \label{Prop:DuallyPseudo}
If $R$ is a quasiorder on a non-empty set $U$, then
both $\mathcal{RS}$ and its dual $\mathcal{RS}^\mathrm{op}$ are
pseudocomplemented lattices.
\end{proposition}

In the next proposition we will determine pseudocomplement operation.

\begin{proposition}
\label{Prop:Pseudo} 
In the lattice $\mathcal{RS}$, for each $X \subseteq U$, we have 
\[
\mathcal{A}(X)^* = \mathcal{A}(X^{\UP \Up c}).
\]
\end{proposition}

\begin{proof}
$\mathcal{A}(X) \wedge  \mathcal{A}(X^{\UP \Up c}) = (\emptyset,\emptyset)$, because 
$X^{\UP \Up c \UP} = X^{\UP \Up \DOWN c} = X^{\UP \Up c}$ and 
$X^\UP \cap X^{\UP \Up c} = \emptyset$ if and only if 
$X^\UP \subseteq X^{\UP \Up}$, and the latter holds trivially. 
For the other part, 
$X^\DOWN \cap X^{\UP \Up c \DOWN} \subseteq X^\UP \cap X^{\UP \Up c}
= \emptyset$.

On the other hand, if $\mathcal{A}(X) \wedge \mathcal{A}(Y) =
(\emptyset,\emptyset)$, then $X^\UP \cap Y^\UP = \emptyset$ and $Y^\UP
\subseteq X^{\UP c}$. This implies $Y \subseteq Y^{\UP
\Down} \subseteq X^{\UP c \Down} =
X^{\UP \Up c}$, from which we get $\mathcal{A}(Y) \leq \mathcal{A}(X^{\UP \Up c})$. 
Thus, $\mathcal{A}(X)^* = \mathcal{A}(X^{\UP \Up c})$.
\end{proof}

Notice that $\mathcal{RS}$ is not necessarily a Stone lattice. In the final
section, we give a condition under which $\mathcal{RS}$ is a Stone lattice.

We conclude this section by describing also dual pseudocomplements in $\mathcal{RS}$.
A \emph{dual pseudocomplement} $x^+$ of $x \in L$ in a lattice $(L,\leq)$ with a greatest
element $1$ is such that  $x \vee x^+ = 1$ and $x \vee y = 1$ implies $x^+ \leq y$
for all $y \in L$.

\begin{proposition}
\label{Prop:DualPseudo} 
In the lattice $\mathcal{RS}$, for each $X \subseteq U$, we have 
\[
 \mathcal{A}(X)^+ = \mathcal{A}(X^{\DOWN \Down c}). 
\]
\end{proposition}

\begin{proof}
$\mathcal{A}(X) \vee \mathcal{A}(X^{\DOWN \Down c})  = (U,U)$, because 
$X^{\DOWN \Down c \DOWN} =
X^{\DOWN \Down \UP c} = X^{\DOWN \Down c}$, and 
$X^{\DOWN \Down} \subseteq X^\DOWN$ implies $X^{\DOWN c} 
\subseteq X^{\DOWN \Down c}$ and $X^\DOWN \cup
X^{\DOWN \Down c} \supseteq X^\DOWN \cup X^{\DOWN c} = U$. 
Similarly, $X^\UP \cup X^{\DOWN \Down c \UP} \supseteq X^\DOWN \cup X^{\DOWN \Down c} = U$.

If $\mathcal{A}(X) \vee \mathcal{A}(Y) = (U,U)$, then $X^\DOWN \cup Y^\DOWN
= U$ and $X^{\DOWN c} \subseteq Y^\DOWN$. This implies 
$X^{\DOWN \Down c} = X^{\DOWN c \Up} \subseteq Y^{\DOWN \Up} \subseteq Y$.
From this we directly obtain $\mathcal{A}(X^{\DOWN \Down c}) \leq \mathcal{A}(Y)$.
\end{proof}

\section{Completely Irreducible Elements} \label{Sec:JoinIrreducibles}

In this section, we find the set of completely join-irreducible elements
of $\mathcal{RS}$, and show how all elements can be represented as a join of
these. Also completely meet-irreducible elements are characterized.
For a complete lattice $L$, an element $x \in L$ is \emph{completely join-irreducible} 
if for every subset $S$ of $L$, $x = \bigvee S$ implies that $x \in S$; see \cite{Ran52}.

Notice that for any Alexandrov topology $\mathcal{T}$, the family
$\mathcal{B_T} = \{ N_\mathcal{T}(x) \mid x \in U \}$ of the neighbourhoods
of the points consists of the completely join-irreducible elements of
the complete lattice $(\mathcal{T},\subseteq)$. This means that for all 
$X \in \mathcal{B_T}$ and $\mathcal{H} \subseteq \mathcal{T}$, 
$X = \bigcup \mathcal{H}$ implies $X = Y$ for some $Y \in \mathcal{H}$.
 
Let us define the set of rough sets 
\begin{equation*}
\mathcal{J} = \{ (\emptyset,\{x\}^\UP) \mid \ |R(x)| \geq 2 \} \cup  \{
(\{x\}^\Up,\{x\}^{\Up \UP}) \mid x \in U\}. 
\end{equation*}
We can write the following lemma.

\begin{lemma}
\label{Lem:JoinIrreducibles} The members of $\mathcal{J}$ are completely
join-irreducible.
\end{lemma}

\begin{proof}
If $|R(x)| \geq 2$, then $R(x) \not \subseteq \{x\}$ and $\{x\}^\DOWN = \emptyset$, 
which implies $\mathcal{A}(\{x\}) = (\emptyset,\{x\}^\UP) \in \mathcal{RS}$. 
Because $\{x\}^\UP$ is  a member of the smallest base $\{ \{x\}^\UP \mid x \in U\}$
of the topology $\mathcal{T}^\UP$, $\{x\}^\UP$ is completely join irreducible
in the complete lattice $(\mathcal{T}^\UP,\subseteq)$. This means that
$\{x\}^\UP = \bigcup_{i \in I} X_i^\UP$ implies $\{x\}^\UP =  X_i^\UP$ for some $i \in I$.
Therefore, the rough set   $\mathcal{A}(\{x\}) = (\emptyset,\{x\}^\UP)$
is completely join-irreducible in $\mathcal{RS}$.

It is clear that for each $x \in U$, $\mathcal{A}(\{x\}^\Up) =
(\{x\}^\Up,\{x\}^{\Up \UP}) \in \mathcal{RS}$, because 
$\{x\}^{\Up \DOWN} = \{x\}^\Up$. We show that $(\{x\}^\Up,\{x\}^{\Up \UP})$ is 
completely join-irreducible. Suppose that there exists a family 
$\{ \mathcal{A}(X_i)\}_{i \in I} = \{ (X_i^\DOWN, X_i^\UP ) \}_{i \in I}$ such that 
\begin{equation*}
(\{x\}^\Up,\{x\}^{\Up \UP}) = \bigvee_{i \in I} \mathcal{A}(X_i)  
= \Big ( \bigcup_{i\in I} X_{i}^\DOWN, \bigcup_{i\in I} X_{i}^\UP \Big ).
\end{equation*}
Since $\{x\}^\Up$ is a member of the smallest base $\{ \{x\}^\Up \mid x \in U\}$
of the topology $\mathcal{T}^\DOWN$, $\{x\}^\Up$ is completely join-irreducible
in the lattice $(\mathcal{T}^\DOWN,\subseteq)$. Therefore,
$\{x\}^\Up =  \bigcup_{i\in I} X_{i}^\DOWN$ implies $\{x\}^\Up = X_i^\DOWN$ for some $i \in I$. 
Since  $X_i^\DOWN \subseteq X_i$, this also gives 
$\{x\}^{\Up \UP} \subseteq X_i^\UP$. The converse, $X_i^\UP \subseteq \{x\}^{\Up \UP}$, 
holds trivially. Thus, $\{x\}^{\Up \UP} = X_i^\UP$ and 
\begin{equation*}
(\{x\}^\Up,\{x\}^{\Up \UP}) = (X_i^\DOWN,X_i^\UP).
\end{equation*}
Hence, $(\{x\}^\Up,\{x\}^{\Up \UP})$ is completely join-irreducible in $\mathcal{RS}$.
\end{proof}

Our next theorem shows that $\mathcal{J}$ is the set of all completely join-irreducible
elements.

\begin{theorem}
\label{Thm:JoinIrreducibles} $\mathcal{J}$ is the set of completely
join-irreducible elements of the complete lattice $\mathcal{RS}$.
Any nonzero element of $\mathcal{RS}$ is a join of some completely
join-irreducible elements of $\mathcal{J}$.
\end{theorem}

\begin{proof}
Next, we will prove that each $(X^\DOWN,X^\UP) \in \textit{RS\/}$ can be
expressed as the join of some elements in $\mathcal{J}$. We begin by showing
that for all $X \subseteq U$, 
\begin{equation*}
X^\DOWN = \bigcup \{ \{x\}^\Up \mid \{x\}^\Up \subseteq X\}. 
\end{equation*}
Because $X^\DOWN \in \mathcal{T}^\DOWN$ and $\{ \{x\}^\Up \mid x \in U\}$ is the
smallest base of the topology $\mathcal{T}^\DOWN$, we get 
\begin{eqnarray*}
X^\DOWN & = & \bigcup \{ \{x\}^\Up \mid x \in X^\DOWN\} \\
& = & \bigcup \{ \{x\}^\Up \mid R(x) \subseteq X \} \\
& = & \bigcup \{ \{x\}^\Up \mid \{x\}^\Up \subseteq X \}.
\end{eqnarray*}
Since $\bigcup \{ \{x\}^\Up \mid \{x\}^\Up \subseteq X \} \subseteq X$ and 
$^\UP$ distributes over unions, we have 
\begin{equation*}
\bigcup \{ \{x\}^{\Up \UP} \mid \{x\}^\Up \subseteq X \}
\subseteq X^\UP. 
\end{equation*}
Hence, 
\begin{equation*}
\bigvee \{ (\{x\}^\Up,\{x\}^{\Up \UP}) \mid \{x\}^\Up
\subseteq X\} \leq (X^\DOWN,X^\UP). 
\end{equation*}
Obviously, 
\begin{equation*}
\bigvee \{ (\emptyset,\{x\}^\UP) \mid x \in X \mbox{ and } |R(x)| \geq 2 \}
\leq (X^\DOWN,X^\UP) 
\end{equation*}
Next we will show that 
\begin{eqnarray*}
(X^\DOWN,X^\UP) & = & \bigvee \{ (\emptyset,\{x\}^\UP) \mid x \in X \mbox{
and } |R(x)| \geq 2 \} \\
& & \vee \ \bigvee \{ (\{x\}^\Up,\{x\}^{\Up \UP}) \mid
\{x\}^\Up \subseteq X\}.
\end{eqnarray*}
Let $x \in X^\UP$. If $|R(x)| = 1$, then clearly also $x \in X^\DOWN$. In
this case, $\{x\}^\Up = \{x\} \subseteq X$, $x \in \{x\}^\Up$, and $x \in
\{x\}^{\Up \UP}$.

If $|R(x)| \geq 2$, then we consider three different cases: (i) $x \in
X^\DOWN$, (ii) $x \in X \setminus X^\DOWN$, and (iii) $x \in X^\UP \setminus X$.

(i) If $x \in X^\DOWN$, then $x \in \{x\}^\Up = R(x) \subseteq X$, and
trivially $x \in \{x\}^{\Up \UP}$.

(ii) If $x \in X$, but $x \notin X^\DOWN$, then $\{x\}^\Up = R(x) \not
\subseteq X$. However, $x \in \{x\}^\UP$ and $x \in X$. In addition, $x \in X
$ and $x \notin X^\DOWN$ imply $|R(x)| \geq 2$ and $\{x\}^\DOWN = \emptyset$.

(iii) If $x \in X^\UP$, but $x \notin X$, then necessarily $x \in \{y\}^\UP$
for some $y \in X$ such that $x \neq y$. We have now two possibilities,
either $R(y) \subseteq X$ or $R(y) \not \subseteq X$. If $\{y\}^\Up = R(y)
\subseteq X$, then necessarily $x \notin \{y\}^\Up$. However, $x \in
\{y\}^\UP \subseteq \{y\}^{\Up \UP}$. If $R(y) \not
\subseteq X$, then $R(y) = \{y\}^\Up$ contains at least two elements,
because $y \in X$. This implies $\{y\}^\DOWN = \emptyset$. Thus, 
$(\emptyset,\{y\}^\UP) \in \mathcal{RS}$ and recall that $x \in \{y\}^\UP$
and $y \in X$.
\end{proof}

Clearly, the image of the set $\mathcal{J}$ under the mapping 
$c \colon \mathcal{A}(X) \mapsto \mathcal{A}(X^c)$ is
\begin{align*}
\mathcal{M} & = \{ \mathcal{A}( \{x\}^c )   \mid \ |R(x)| \geq 2\} \cup \{  \mathcal{A}(\{x\}^{\Up c} ) \mid x \in U\} \\
    & = \{ ( \{x\}^{\UP c}, U )  \mid \ |R(x)| \geq 2\} \cup \{ ( \{x\}^{\Up \UP c}, \{x\}^{\Up c} ) \mid x \in U\} .
\end{align*}
Since the completely meet-irreducible elements of the lattice $\mathcal{RS}$ are
just the completely join-irreducible elements of its dual $\mathcal{RS}^\mathrm{op}$,
and because $\mathcal{RS}$ is order-isomorphic to $\mathcal{RS}^\mathrm{op}$ via the mapping
$\mathcal{A}(X) \mapsto \mathcal{A}(X^c)$, we obtain the following corollary.

\begin{corollary} \label{Cor:MeetIrreducibles}
 $\mathcal{M}$ is the set of completely
meet-irreducible elements of the complete lattice $\mathcal{RS}$.
Any nonunit element of $\mathcal{RS}$ is a meet of some completely
meet-irreducible elements of $\mathcal{M}$.
\end{corollary}

\section{Characterization of the Stonean Case} \label{Sec:Stonean}

In the case of a quasiorder $R\subseteq U\times U$, the smallest equivalence
containing $R$ is $R\vee R^{-1}$, where $\vee $ denotes the join
in the lattice  of all quasiorders on $U$ ordered with the set-inclusion relation $\subseteq$.
Furthermore,  $R\vee R^{-1}$ is equal to the transitive closure of the relation $R\cup R^{-1}$. 
Hence, the connected components of $R$ are just the equivalence classes of 
$R\vee R^{-1}$.

\begin{proposition} \label{Prop:ACC}
Let $R$  be a quasiorder on $U$.
Then, the following assertions are equivalent for any $X \subseteq U$:
\begin{enumerate}[\rm (i)]
\item $\mathcal{A}(X)$ is a complemented element of $\mathcal{RS}$;

\item $\mathcal{A}(X)$ is an exact element of $\textit{RS\/}$;

\item $X$ is a union of some equivalence classes of $R\vee R^{-1}$.
\end{enumerate}
\end{proposition}

\begin{proof}
(i)$\Rightarrow$(ii): Assume that there exists a set $Y \subseteq U$ such that 
$\mathcal{A}(Y)$ is the complement of $\mathcal{A}(X)$. Then, $X^\UP \cap Y^\UP = \emptyset$ and 
$X^\DOWN\cup Y^\DOWN=U$. Thus, we have $Y^{c} \subseteq Y^{\DOWN c}\subseteq
X^\DOWN\subseteq X \subseteq X^\UP \subseteq Y^{\UP c}\subseteq Y^{c}$, 
proving $X^{\DOWN }=X=X^\UP=Y^{c}$.

(ii)$\Rightarrow$(iii): Clearly, (ii) implies $X^\DOWN = X$ and $X^{\Down} = X^{\UP\Down} = X^\UP = X$.
Hence, for any $x\in X$, we have $(R\cup R^{-1})(x) \subseteq X$, that is, $X$ 
is closed with respect to the relation $R\cup R^{-1}$. Then, $X$ must also be closed
with respect to the transitive closure $R\vee R^{-1}$ of $R\cup R^{-1}$, that is,
$(R\vee R^{-1})(x) \subseteq X$ for all $x \in X$. This implies
$X = \bigcup \{ (R\vee R^{-1})(x) \mid x \in X\}$.

(iii)$\Rightarrow$(i): Suppose that $X = \bigcup \mathcal{H}$,
where $\mathcal{H}$ is a set of some equivalence classes of $R\vee R^{-1}$. 
Because for each $C \in \mathcal{H}$, the set $C$ is a connected component of $R$, $X^{\DOWN }=X^\UP=X$
by Corollary~\ref{Cor:Exact}. Then, $X^{c\DOWN} = X^{\UP c} = X^{c} = X^{\DOWN c} = X^{c\UP}$, that is,
also  $\mathcal{A}(X^{c})$ is an exact element of $\textit{RS\/}$. Since
$\mathcal{A}(X) \wedge \mathcal{A}(X^{c}) = (X \cap X^{c}, X\cap X^{c}) =
(\emptyset,\emptyset)=\mathcal{A}(\emptyset)$ \ and \ 
$\mathcal{A}(X) \vee \mathcal{A}(X^{c}) = (X\cup X^{c},X\cup X^{c}) =(U,U)= \mathcal{A}(U)$,
we have that $\mathcal{A}(X)$ is a complemented element of $\mathcal{RS}$.
\end{proof}

\begin{remark} \label{Rem:Exact}
If the assumption of Proposition~\ref{Prop:ACC} is satisfied, then from the arguments 
of the above proof it follows also that $\mathcal{A}(X)$ is exact  if and only if 
$\mathcal{A}(X^{c})$ is exact.
\end{remark}

An element $a$ of a bounded lattice $\mathcal{L} = (L,\leq)$ is called a 
\emph{central element} of $\mathcal L$ if $a$ is complemented and for all 
$x,y\in L$ the sublattice generated by $\{a,x,y\}$ is distributive. 
Notice that the complement of a central element is unique and it is also a 
central element of $\mathcal L$. Clearly, the least element $0$
and the greatest element $1$ of $\mathcal L$ are always central elements. 
It is known that $\mathcal{L} \cong \mathcal{L}_{1}\times \mathcal{L}_{2}$ 
for some nontrivial bounded lattices $\mathcal{L}_{1}$ and $\mathcal{L}_{2}$ 
if and only if there exists a pair of central elements $c_{1}, c_{2}\in L \setminus \{0,1\}$ 
such that $(c_{1}]\cong L_{1}$, $(c_{2}]\cong L_{2}$ and $c_{1}$ and $c_{2}$ are complements 
of each other, where $(x]$ denotes the principal ideal $\{ y \in L \mid y \leq x \}$ of $x$
(for details, see  e.g. \cite{Rade00,Rade03}).
A lattice $\mathcal L$ is \emph{directly indecomposable} if there are no nontrivial 
lattices $\mathcal{L}_{1}$ and $\mathcal{L}_{2}$ satisfying 
$\mathcal{L} \cong \mathcal{L}_{1} \times \mathcal{L}_{2}$ (see \cite{BuSa81},
for instance). Clearly, this is equivalent to the fact that $\mathcal L$ has 
no nontrivial central elements. It is
also obvious that the central elements of a bounded distributive lattice are 
exactly its complemented elements.

\begin{proposition} \label{Prop:DirectlyIrreducible}
The following assertions are true for any quasiorder $R$.

\begin{enumerate}[\rm (i)]

\item For any connected component $C \in \mathfrak{Co}$, the lattice $\mathcal{RS}(C)$ is directly indecomposable.

\item The lattice $\mathcal{RS}$ is directly  indecomposable if and only if 
$R$ is a connected quasiorder, that is, $R$ has a single connected component.
\end{enumerate}
\end{proposition}

\begin{proof} (i) Assume that there exists $C \in \mathfrak{Co}$ such that
the lattice $\mathcal{RS}(C)$ is directly decomposable. Then,
$\mathcal{RS}(C)$ has at least one nontrivial central element.
This means that in $\mathcal{RS}(C)$ exists a complemented element 
$\mathcal{A}(X)$ for some $X \subseteq C$ such that 
$\mathcal{A}(X) \neq (\emptyset,\emptyset)$ and $\mathcal{A}(X) \neq (C,C)$. 
Then, according to Proposition~\ref{Prop:ACC}, $X$ is a join of some equivalence 
classes of the restriction of $R\vee R^{-1}$ to $C$. However $(x,y) \in R\vee R^{-1}$ 
is satisfied for all $x,y\in C$, because $C$ is an equivalence class
of $R\vee R^{-1}$. This fact implies $X = C$, that is, 
$\mathcal{A}(X)=(C,C)$, a contradiction.

(ii) If $R$ is a connected quasiorder on $U$, then $R\vee R^{-1}$
has just one equivalence class $U$. Therefore, by applying
Proposition~\ref{Prop:ACC}, the complemented elements are just $\mathcal{A}(\emptyset)$
and $\mathcal{A}(U)$. This means that $\mathcal{RS}$ contains only the trivial central elements 
$\mathcal{A}(\emptyset)$  and  $\mathcal{A}(U)$ that are the least and the
greatest elements of $\mathcal{RS}$, respectively. Thus, the lattice $\mathcal{RS}$ 
is directly  indecomposable.

The other part is an obvious consequence of (i) and the isomorphism 
of $\mathcal{RS}$ and $\prod_{C \in \mathfrak{Co}} \mathcal{RS}(C)$ established in 
Theorem~\ref{Thm:Representation}.
\end{proof}

Let $\mathcal L$ be a pseudocomplemented bounded distributive lattice. 
If $x^{\ast} \vee x^{\ast\ast} = 1$ holds for all $x\in L$, then $\mathcal L$ is called a 
\emph{Stone lattice}. Obviously, this is equivalent to the fact that 
$x^*$ is a complemented element of $\mathcal L$ for each $x \in L$.

\begin{theorem} \label{Thm:Characterization}
Let $R$ be a quasiorder on $U$.
Then, $\mathcal{RS}$ is a Stone lattice \  if and only if 
\ $R^{-1}\circ R = R\vee R^{-1}$.
\end{theorem}

\begin{proof}
Assume that $\mathcal{RS}$ is a Stone lattice. Then, for any 
$X \subseteq U$, the rough set $\mathcal{A}(X)^{\ast} = \mathcal{A}(X^{\UP \Up c})$ 
is a complemented element of $\mathcal{RS}$. This implies by
Proposition~\ref{Prop:ACC} and Remark~\ref{Rem:Exact} that
$\mathcal{A}(X^{\UP \Up c})$ and also its complement $\mathcal{A}(X^{\UP \Up })$ are 
exact elements of $\textit{RS\/}$. Let $x$ be an arbitrary element of $U$ and 
let us set $X = \{x\}$. Then, it is
easy to see that $X^{\UP \Up }=(R^{-1}\circ R)(x)$. By
Proposition~\ref{Prop:ACC}, $(R^{-1}\circ R)(x)$ is equal to the union of
some equivalence classes of $R \vee R^{-1}$. Because 
$R^{-1}\circ R \subseteq R\vee R^{-1}$, this implies 
$(R^{-1}\circ R)(x) = (R\vee R^{-1})(x)$. As this equality is satisfied for all 
$x\in U$, we obtain $R^{-1}\circ R = R\vee R^{-1}$.

Conversely, assume that the condition of Theorem~\ref{Thm:Characterization}
and the equality $R^{-1}\circ R = R\vee R^{-1}$ are satisfied. Then,
according to Corollary~\ref{Cor:MeetsJoins} and Proposition~\ref{Prop:DuallyPseudo},  
$\mathcal{RS}$ is a distributive pseudocomplemented complete lattice. 
We have to show that $\mathcal{A}(X)^*$ is complemented for any $X \subseteq U$.

Assume that $X\subseteq U$. Let $x\in X^{\UP \Up }$ and $y\in (R\vee R^{-1})(x)$. 
Then, by the definition of $X^{\UP \Up }$, 
there exist $z \in R^{-1}(x)\cap X^{\UP }$ and $v \in R(z)\cap X$. 
These mean $x \inR{R^{-1}} z$ and $z \inR{R} v$, from which we obtain 
$(x,v) \in R^{-1}\circ R$ with $v \in X$. Therefore, $(v,x) \in R\vee R^{-1}$, 
and so $(x,y) \in R\vee R^{-1}$ implies $(v,y) \in R\vee R^{-1}$. This
means that  $y\in (R\vee R^{-1})(v) = (R^{-1}\circ R)(v) = \{v\}^{\UP \Up }\subseteq X^{\UP \Up }$. 
This result proves that 
$(R\vee R^{-1})(x) \subseteq X^{\UP \Up }$ for all $x\in X^{\UP \Up }$.
From this we get that $X^{\UP \Up }$ is the union of some classes of $R\vee R^{-1}$. 
So, by Proposition~\ref{Prop:ACC}, $\mathcal{A}(X^{\UP \Up })$ is 
a complemented and exact element of $\mathcal{RS}$. Then, 
by Remark~\ref{Rem:Exact}, also $\mathcal{A}(X)^{\ast} = \mathcal{A}(X^{\UP \Up c})$ 
is an exact and complemented element of $\mathcal{RS}$.
Therefore, $\mathcal{RS}$ is a Stone lattice.
\end{proof}

A partially ordered set $(P,\leq)$ is called \emph{down-directed}, if for
any $a,b\in P$ there exists $c\in P$ with $c \leq a,b$. In what follows, we
deduce some corollaries of the above theorem in cases $R$ is a partial order or an equivalence,
respectively. Notice that case (i) of the next corollary shows that 
the result of M.~Gehrke and E.~Walker stating that for equivalences, 
$\mathcal{RS}$ is isomorphic two and three elements also follows in an alternative 
way from our Theorem~\ref{Thm:Characterization}.
We note that the details on the direct decomposition of complete Stone lattices
can be found in \cite{Rade00}.

\pagebreak%

\begin{corollary} \label{Cor:Generalization}
Let $R$ be a binary relation on $U$.
\begin{enumerate}[\rm (i)]
\item If $R$ is an equivalence, then $\mathcal{RS}$ is a completely
distributive Stone lattice which is isomorphic to a direct product of chains of 
two and three elements \cite[Theorem 2]{GeWa92}. 

\item If $R$ is a partial order, then $\mathcal{RS}$ is a Stone lattice 
if and only if any connected component of $(U,R)$ is down-directed.
\end{enumerate}
\end{corollary}

\begin{proof}
(i) The fact that $\mathcal{RS}$ is a completely distributive lattice
follows from Corollary~\ref{Cor:MeetsJoins}. 
Additionally, $R^{-1} \circ R = R\vee R^{-1}$ because $R$ is an equivalence. 
Hence, Theorem~\ref{Thm:Characterization}
implies that $\mathcal{RS}$ is a Stone lattice. In view of
Theorem~\ref{Thm:Representation}, we have that $\mathcal{RS}$ and $\prod_{C \in \mathfrak{Co}}\mathcal{RS}(C)$
are isomorphic, where $\mathfrak{Co}$ is now the set of the equivalence classes of $R$. 
If $C \in \mathfrak{Co}$ consists of a single element $a$, then 
$\textit{RS\/}(C) = \{(\emptyset,\emptyset ), (\{a\},\{a\})\}$ 
and $\mathcal{RS}(C)$ is a chain of two elements. 
Similarly, if $|C| \geq 2$, then  
$\textit{RS\/}(C) = \{ (\emptyset ,\emptyset), (\emptyset,C), (C,C) \}$
and $\mathcal{RS}(C)$ is a chain of three elements. Therefore, $\mathcal{RS}$
is of the form $\mathbf{2}^I \times \mathbf{3}^J$,
where $I \cup J = \mathfrak{Co}$ and $I \cap J = \emptyset$. As earlier, $I$ stands for
singleton equivalence classes and $J$ denotes non-singleton $R$-classes.

(ii) As the connected components in $\mathfrak{Co}$ are the equivalence classes of 
$R \vee R^{-1}$, it is easy to check that $R^{-1} \circ R = R \vee R^{-1}$ 
if and only if for any $C \in \mathfrak{Co}$ and $a,b\in C$, there exists $c\in C$ with 
$c \inR{R} a$ and $c\inR{R}b$. This means that all connected components $C$ are
down-directed with respect to the partial order $R$ restricted to $C$.
\end{proof}

\subsection*{Acknowledgement}
We would like to take the opportunity to thank the
anonymous referee whose insightful comments and
suggestions helped us to improve the paper significantly.

\providecommand{\bysame}{\leavevmode\hbox to3em{\hrulefill}\thinspace}
\providecommand{\MR}{\relax\ifhmode\unskip\space\fi MR }
\providecommand{\MRhref}[2]{%
  \href{http://www.ams.org/mathscinet-getitem?mr=#1}{#2}
}
\providecommand{\href}[2]{#2}

%\bibliographystyle{amsplain}
%\bibliography{Rough}

\end{document}